\documentclass[12pt]{amsart}

\usepackage{amsmath,amssymb,amsbsy,amsfonts,amsthm,latexsym,
         amsopn,amstext,amsxtra,euscript,amscd,amsthm, mathtools, dsfont}
\usepackage[utf8]{inputenc}
\usepackage[margin=1.5in]{geometry}

\allowdisplaybreaks

\numberwithin{equation}{section}

\newtheorem{lem}{Lemma}
\newtheorem{lemma}[lem]{Lemma}

\newtheorem{thm}{Theorem}

\theoremstyle{remark}
\newtheorem*{remark}{Remark}

\begin{document}
\title{
% A NEW METHOD TO SOLVE BROKEN STICK PROBABILITY PROBLEMS
%A new method to solve broken stick type/like (probability) problems}
%On the probability of forming polygons/$k$-gons from a ($n$-piece) broken stick}
On the probability of forming polygons from a broken stick}
\author{William Verreault}

\address{
William Verreault \\
D\'epartement de math\'ematiques et de statistique \\ %\hfill (Received 00 00 2010)\\
Universit\'e Laval   \\ %\hfill (Revised  00 00 2010)\\
Qu\'ebec G1V 0A6, Canada}
\email{william.verreault.2@ulaval.ca}

\begin{abstract}
     Break a stick at random at $n-1$ points to obtain $n$ pieces. 
     We give an explicit formula for the probability that every choice of $k$ segments from this broken stick can form a $k$-gon, generalizing similar work.
     The method we use can be applied to other geometric probability problems involving broken sticks, which are part of a long-standing class of recreational probability problems with several applications to real-world models. 
\end{abstract}

\maketitle

\section{Introduction}
The setup of a broken stick problem is to break a stick at random at $n-1$ points to obtain $n$ pieces. Then one can ask geometric probability questions involving the broken stick.
It is clear that this setup is equivalent to choosing $n-1$ numbers at random from the interval $(0,L)$, where $L>0$ is the length of the stick. Normalizing the stick to be of length one if needed, we may as well restrict our attention to the interval $(0,1)$. 
 But without further considerations, the probabilities one seeks in a broken stick problem seem nearly impossible to calculate. This point of view, however, shows that the broken stick problems are closely related to the random division of an interval, which has been extensively studied in the past (see \cite{darling_class_1953, holst_lengths_1980, pyke_spacings_1965,steutel_random_1967} for concrete examples). Interpreting random divisions of an interval as a Poisson point process has already been a prominent idea in some of these previous work, and it is the method we adapt here for broken stick problems.

As it was pointed out in the project reports \cite{kong_random_nodate,page_calculus_nodate} and in Darling's article \textit{On a class of problems related to the random division of an interval} \cite{darling_class_1953}, on top of being the source of interesting recreational probability problems, the broken stick model has applications in various areas such as the propagation of infectious diseases, traffic flow, finance, etc., and is a good match for a variety of real-world data sets. 
    The model of the broken stick is also used in statistical science \cite{ghent_application_1968} and in machine learning and can be used to develop new sampling methods.

    In this paper, for fixed integers $n\geq k\geq 3$, we let $\mathbb{P}_{k,n}$ stand for the probability of being able to form a $k$-gon with every choice of $k$ segments from the $n$ pieces of a broken stick. We will present a general approach to find these broken stick-type probabilities, its systematic application being rather straightforward from one problem to the next, while the combinatorial or geometric approaches one might think of at first have obvious limitations. 
 The main result we will prove using our method is the following, the proof of which will appear in Section \ref{proof}.
    \begin{thm}{}
    Let $n\geq k\geq 3$ be positive integers. The probability that every choice of $k$ segments from a stick broken up into $n$ parts will form a $k$-gon is given by
$$
\mathbb{P}_{k,n}=\frac{n(n-1)\cdots(n-k+3)}{n-k+2}\sum_{j=1}^{n-k+2}\frac{(-1)^{j+1}}{j^{k-3}}\binom{n-k+2}{j}\left(\frac{n-k+2}{j}+1\right)_{k-2}^{-1},
$$
where  $$\left(\frac{n-k+2}{j}+1\right)_{k-2}=\prod_{i=0}^{k-3}\left(\frac{n-k+2}{j}+1+i\right)$$ stands for the rising factorial function or Pochhammer symbol.
    \end{thm}

We now mention some of the other works that have covered broken stick problems and give a brief summary of the history of these problems.
    Finding the probability of forming a triangle from a stick broken at random in three segments is a classic probability problem with answer $1/4$ that originated from the Mathematical Tripos. Although very geometric in nature, this problem was first approached with combinatorial techniques until Poincar\'e gave a geometric proof in his treatise \textit{Calcul des probabilit\'es} (to learn more about this, see \cite{goodman_problem_2008}).
    A natural generalization of this problem involving polygons already appeared in $1866$ in the \textit{Educational Times} \cite{clifford_problem_1866} : \textit{A line of length} $a$ \textit{is broken up into $n$ pieces at random; prove that the chance that they cannot be made into a polygon of} $n$ \textit{sides is} $n2^{1-n}$. A popular solution to this problem was given by D'Andrea and Gómez \cite{dandrea_broken_2006} in their note \textit{The Broken Spaghetti Noodle}, where they show that the probability of being able to form an $n$-gon out of the $n$ pieces of a broken stick is given by $1-n/2^{n-1}$. It is worth noting that this problem has a reformulation involving circles that can be stated as a question as follows : \textit{If $n$ points are randomly selected on the circumference of a circle, what is the probability that they will all fall within some semicircle?}
    Next, we would like to generalize these results by using $3\leq k\leq n$ pieces to form a $k$-gon out of the $n$ pieces.
    % After working out this problem, immediate generalizations come to mind. For instance, breaking the stick up in $n$ parts to create regular polygons.
    % Then, for $n$ big enough, we consider creating more than one triangle. Even more interesting, we would like to form other regular polygons out of our broken stick.
    These problems then turn out to be much harder than expected and only a few have found an answer. In research project reports of the Illinois Geometry Lab \cite{kong_random_nodate, page_calculus_nodate}, two extreme cases were considered. The authors investigated the probability that given a stick broken up into $n$ pieces, there exist three segments that can form a triangle, as well as the probability that all choices of three segments can form a triangle. It was stated without proof that the answers are respectively $$1-n!\prod_{j=2}^n (F_{j+2}-1)^{-1} \text{ and } \binom{2n-2}{n}^{-1},$$ where $F_j$ is the $j$th Fibonacci number. 
    % D'andrea and Gómez \cite{D'andrea-Gómez} answered a particular case of the generalization involving polygons in their note \textit{The broken spaghetti noodle}, that is the probability of being able to form an $n$-gon out of the $n$ pieces of a broken stick. 
    Crowdmath's $2017$ project\footnote{Crowdmath is an open project for high school and college students to collaborate on large research projects and is based on the model of Polymath. $2017$'s project was called \textit{The Broken Stick Problem}.} answered many broken stick type questions, but one of the only problems they could not solve was on forming polygons from a broken stick. They mention that this particular problem generalizes many of the results they obtained.
    % Unfortunately, there are no other similar results known.
    More recently, the authors in \cite{petersen_broken_2020} have considered a discrete variant of the \textit{Broken Spaghetti Noodle} problem where they impose that the stick has integer length and can only be broken at integer increments.

    The formula given by Theorem $1$ is readily seen to simplify according to known results when $k=n$:
    \begin{align*}
        \mathbb{P}_{n,n}&=\frac{n(n-1)\cdots3}{2}\sum_{j=1}^2\frac{(-1)^{j+1}\binom{2}{j}}{j^{n-3}(2/j+1)_{n-2}}\\
        &=\frac{n!}{4}\left(\frac{2}{(3)_{n-2}}-\frac{1}{2^{n-3}(2)_{n-2}}\right).
    \end{align*}
    Since $(3)_{n-2}=n!/2$ and $(2)_{n-2}=(n-1)!$, we get
    \begin{align*}
        \mathbb{P}_{n,n}=\frac{n!}{4}\left(\frac{4}{n!}-\frac{4n}{2^{n-1}n!}\right)=1-\frac{n}{2^{n-1}}.
        \end{align*}
    It is not as obvious that when $k=3$, our formula agrees with the known results, but it will be shown that $\mathbb{P}_{3,n}=\binom{2n-2}{n}^{-1}$ holds and is also equal to the formula given by Theorem $1$ when $k=3$. 
    
    The present paper is organized as follows. A few preliminary results are presented in Section 2 along with the explanation and setup of the method. Theorem 1 is
proved in Section 3. More applications and open problems will be discussed in the Concluding remarks.

\section{Setting up the method}
\subsection{Notation}
   We use $X_1,\ldots, X_{n-1}$ to denote $n-1$ random and independent variables uniformly distributed on the interval $(0,1)$, and $Y_1,\ldots,Y_{n}$ for $n$ random and independent variables with exponential distribution of mean $1$.
   For a given set of random variables, say $Z_1,\ldots,Z_n$, we use $Z_{(k)}$ to denote the $k$th order statistic, that is the $k$th value in order of magnitude. In particular, $Z_{(1)}=\min\{Z_1,\ldots,Z_{n}\}$ and $Z_{(n)}=\max\{Z_1,\ldots,Z_{n}\}$. 
    Then, we let $\Delta_{1}, \Delta_{2},\ldots, \Delta_{n}$ be the differences of the $X_i$, that is
    $$\Delta_{i}=X_{(i)}-X_{(i-1)},$$ with $X_{(0)}$ and $X_{(n)}$ defined to be $0$ and $1$ respectively. These are often referred to as spacings.
    Finally, we use $\overset{d}{=}$ to denote equality in distribution.

\subsection{Preliminary results}
    The basis of our method relies on an observation made by A. Rényi in his seminal paper in the field of order statistics, simply called \textit{On the theory of order statistics} \cite{renyi_theory_1953}. He observed that 
    if we set $W_k=Y_1+\cdots+Y_k$ for $1\leq k\leq n$, then
    \begin{equation} \label{Renyi}
        (X_{(1)},X_{(2)},\ldots,X_{(n-1)})\overset{d}{=}(\frac{W_1}{W_n},\frac{W_2}{W_n},\ldots,\frac{W_{n-1}}{W_n}).
        \end{equation}
    This roughly says that the $n-1$ normalized exponential sums $W_1/W_n,\ldots,W_{n-1}/W_n$ in the interval $(0,1)$ behave like $n-1$ randomly selected points in $(0,1)$. 
    As it is well known (see \cite{pyke_spacings_1965} for instance), \eqref{Renyi} is in fact equivalent to the differences $\Delta_{1}, \ldots, \Delta_{n}$ and the variables $Y_1/W_n,\ldots,Y_n/W_n$ having the same distribution. Notice $W_n$ is gamma$(n,1)$ distributed and independent from $Y_i/W_n$, whence we also have the following.
    \begin{lemma}\label{lem main} With the random variables defined as before, we have
    $$(\Delta_{(1)}, \Delta_{(2)},\ldots, \Delta_{(n)})\overset{d}{=}(\frac{Y_{(1)}}{W_n},\frac{Y_{(2)}}{W_n},\ldots,\frac{Y_{(n)}}{W_n}).$$
    \end{lemma}
    
    The last result we state gives us a way to write $\mathbb{P}_{k,n}$ in terms of the $Y_{(i)}$ and is the starting point of the proof of Theorem $1$.
    \begin{lem} \label{lem triangle}
    The probability $\mathbb{P}_{k,n}$ is given by $\mathbb{P}(Y_{(n)} \leq Y_{(1)}+\cdots+Y_{(k-1)})$.
    \end{lem}
    \begin{proof}
    A well known generalized triangle inequality for polygons says that for a set of lengths $x_1,x_2,\ldots,x_k$ to form a $k$-gon, it must be that for $i=1,\ldots,k$, \begin{equation*} x_i\leq\sum_{j\neq i}x_j.\end{equation*}
    This is equivalent to $$\max_{1\leq i\leq k}\{x_i\}\leq x_1+\cdots+x_k-\max_{1\leq i\leq k}\{x_i\}.$$
    Since the lengths $x_i$ are in our case represented by $\Delta_i$, we have that $\Delta_{(n)} \leq \Delta_{(1)}+\cdots+\Delta_{(k-1)}$ implies we can form a $k$-gon with $\Delta_{(1)},\ldots,\Delta_{(k-1)},\Delta_{(n)}$. But since we have put together the $k-1$ smallest pieces and the biggest one, it follows that all other choices of $k$ segments from $\Delta_{(1)},\ldots,\Delta_{(n)}$ will form a $k$-gon. Thus,
    $$
    \mathbb{P}_{k,n}=\mathbb{P}(\Delta_{(n)} \leq \Delta_{(1)}+\cdots+\Delta_{(k-1)}).
    $$
    By Lemma \ref{lem main}, this is equivalently given by $\mathbb{P}(Y_{(n)}/W_k \leq Y_{(1)}/W_k+\cdots+Y_{(k-1)}/W_k)$. Getting rid of the common factor, we are looking for $\mathbb{P}(Y_{(n)} \leq Y_{(1)}+\cdots+Y_{(k-1)})$.
    \end{proof}
%   It is important to note that the only major difference one might encounter in using this method to solve a different broken stick problem would come from choosing how to apply a similar Lemma in another problem. Such a case(...) is given as an example / treated in the Concluding remarks.
\begin{remark}
Under this setup, we can give a quick proof that $\mathbb{P}_{n,n}=1-n/2^{n-1}$ which has not appeared in the literature.
\end{remark}
\begin{proof}
The proof of Lemma \ref{lem triangle} reveals that $\mathbb{P}_{n,n}=\mathbb{P}(\Delta_{(n)}\leq 1/2)$. The distribution of the maximum of the order statistics of these spacings has been well-studied. In fact, the survivor function of the biggest spacing is \begin{equation} \label{whit}
        \mathbb{P}(\Delta_{(n)}>x)=\sum_{j\geq 1, jx<1} (-1)^{j+1}(1-jx)^{n-1}\binom{n}{j},
    \end{equation} 
        a formula going back to Whitworth \cite[Exercise 667]{whitworth_choice_1901}.
        Now plug in $x=1/2$ to get $n/2^{n-1}$, which gives the result. 
\end{proof}
We provide a quick proof of \eqref{whit} for completeness. It is well-known that for non negative numbers $c_1,\ldots,c_n$,
        $$
    \mathbb{P}(\Delta_{1}>c_1,\ldots,\Delta_{n}>c_n)=\begin{cases}(1-\sum_{i=1}^n c_i)^{n-1} &\mbox{if } \sum c_i<1, \\ 0 &\mbox{otherwise}.
    \end{cases}
    $$
    Then the formula follows using the inclusion-exclusion principle.
    \begin{align*}
    \mathbb{P}(\Delta_{(n)}>x)=\mathbb{P}\left(\bigcup_{i=1}^n \Delta_{i}>x\right) 
    &=\sum_i \mathbb{P}(\Delta_i>x) - \sum_{j<i}\mathbb{P}(\Delta_i>x,\Delta_j>x)+\cdots \\
    &= \sum_{j\geq 1, jx<1} (-1)^{j+1}(1-jx)^{n-1}\binom{n}{j}.
    \end{align*}

\section{Proofs of main results} \label{proof}
We start with a Lemma of independent interest that gives an inequality that will show up in the proof of Theorem $1$.    

\begin{lemma} \label{lemma}
The following equality holds for all $k\geq 4$.
\begin{align*}
    \int_0^{\infty}e^{-(n-k+2)x_{k-2}}\int_{0}^{x_{k-2}}&\cdots\int_{0}^{x_3}e^{-j(2x_2+x_3+\cdots+x_{k-2})}-e^{-j(x_2+\cdots+x_{k-2})}
    \,d{x_2}\cdots d{x_{k-2}}\\
    &=\frac{-1}{j^{k-3}}\left(\frac{n-k+2}{j}+1\right)_{k-2}^{-1},
\end{align*}
with $\left(\frac{n-k+2}{j}+1\right)_{k-2}$ the rising factorial defined in the Introduction.
\end{lemma}

\begin{proof}
To start, we will prove that
\begin{equation}\label{1/(k-2)!}
\sum_{i=2}^{k-1}\frac{(-1)^i}{(i-1)!(k-1-i)!}=\frac{1}{(k-2)!}.
\end{equation}
To see this, rewrite the left hand side as $$\frac{1}{(k-2)!}\sum_{i=2}^{k-1}\binom{k-2}{i-1}(-1)^i.$$
Upon reindexing the sum, we see it is equal to 
$$
\frac{1}{(k-2)!}\left(-\sum_{i=0}^{k-2}\binom{k-2}{i}(-1)^i+1\right).
$$
By the binomial theorem, $-\sum_{i=0}^{k-2}\binom{k-2}{i}(-1)^i=-(1-1)^{k-2}=0$, so the result holds.

Next we will prove by induction on $k$ that for $k\geq 4$,
\begin{align}
    \int_{0}^{x_{k-2}}&\cdots\int_{0}^{x_3}e^{-j(2x_2+x_3+\cdots+x_{k-2})}-e^{-j(x_2+\cdots+x_{k-2})} \nonumber
    \,d{x_2}\cdots d{x_{k-3}}\\
    &=\frac{1}{j^{k-4}}\sum_{i=1}^{k-2}\frac{(-1)^i}{(i-1)!(k-2-i)!}e^{-j(ix_{k-2})} \label{induction}
\end{align}
The case $k=4$ is somewhat of a degenerate case as $e^{-2jx_2}-e^{-jx_2}$ is already under this form. If this is not satisfactory, it is readily seen that when $k=5$,
\begin{equation*}
    \int_0^{x_3}e^{-j(2x_2+x_3)}-e^{-j(x_2+x_3)}
    \,d{x_3}
    =\frac{1}{j}\left(-\frac{1}{2}e^{-j(x_3)}+e^{-j(2x_3)}-\frac{1}{2}e^{-j(3x_3)}\right).
\end{equation*}
Now suppose the result holds for some $k\geq 4$. We then get
\begin{align*}
    \int_{0}^{x_{k-1}}&\cdots\int_{0}^{x_2}e^{-j(2x_2+x_3+\cdots+x_{k-1})}-e^{-j(x_2+\cdots+x_{k-1})}
    \,d{x_2}\cdots d{x_{k-2}}\\
    &=\int_0^{x_{k-1}}e^{-jx_{k-1}}\frac{1}{j^{k-4}}\sum_{i=1}^{k-2}\frac{(-1)^i}{(i-1)!(k-2-i)!}e^{-j(ix_{k-2})}\,d{x_{k-2}}\\
    &=e^{-jx_{k-1}}\frac{1}{j^{k-3}}\sum_{i=1}^{k-2}\frac{(-1)^{i+1}}{i!(k-2-i)!}(e^{-j(ix_{k-1})}-1)\\
    &=\frac{1}{j^{k-3}}\sum_{i=2}^{k-1}\frac{(-1)^{i}}{(i-1)!(k-2-(i-1))!}(e^{-j(ix_{k-1})}-e^{-jx_{k-1}})\\
    &=\frac{1}{j^{k-3}}\sum_{i=2}^{k-1}\frac{(-1)^{i}}{(i-1)!(k-1-i))!}e^{-j(ix_{k-1})} - \frac{e^{-jx_{k-1}}}{j^{k-3}(k-2)!}\\
    &=\frac{1}{j^{k-3}}\sum_{i=1}^{k-1}\frac{(-1)^{i}}{(i-1)!(k-1-i)!}e^{-j(ix_{k-1})}.
\end{align*}
The second to last equality holds by \eqref{1/(k-2)!}, so that the result is proved by induction.

Now we multiply \eqref{induction} by $e^{-(n-k+2)x_{k-2}}$ and integrate from $0$ to $\infty$ with respect to $d{x_{k-2}}$ to get
\begin{align*}
    &\int_0^{\infty}e^{-(n-k+2)x_{k-2}}
    \frac{1}{j^{k-4}}\sum_{i=1}^{k-2}\frac{(-1)^i}{(i-1)!(k-2-i)!}e^{-j(ix_{k-2})}
    \,d{x_{k-2}}\\
    &=\frac{1}{j^{k-4}}\sum_{i=1}^{k-2}\frac{(-1)^i}{(i-1)!(k-2-i)!(n-k+2+ij)}.
\end{align*}
To finish the proof, it suffices to prove that 
$$
\sum_{i=1}^{k-2}\frac{(-1)^i}{(i-1)!(k-2-i)!(n-k+2+ij)}=\frac{-\Gamma\left(1+\frac{n-k+2}{j}\right)}{j\Gamma\left(k-1+\frac{n-k+2}{j}\right)},
$$
since the identity
$
   \Gamma(x+n)/\Gamma(x)=(x)_n
$
for $x$ and $n$ real numbers such that $x$ is not a negative integer is equivalent to the definition of the rising factorial.

To prove this, rewrite the left hand side as
$$
\frac{1}{j(k-2)!}\sum_{i=0}^{k-2}(-1)^i\binom{k-2}{i}\frac{i}{i+\frac{n-k+2}{j}}
$$
and observe 
$$\frac{i}{i+\frac{n-k+2}{j}}=1-\frac{n-k+2}{j}\int_{0}^1t^{\frac{n-k+2}{j}+i-1}\,dt.$$
Because $\sum_{i=0}^{k-2}(-1)^i\binom{k-2}{i}=(1-1)^{k-2}=0$, we are left with evaluating
$$
-\frac{\frac{n-k+2}{j}}{j(k-2)!}\int_0^1t^{\frac{n-k+2}{j}-1}\sum_{i=0}^{k-2}\binom{k-2}{i}(-t)^i\,dt,
$$
and since $\sum_{i=0}^{k-2}\binom{k-2}{i}(-t)^i=(1-t)^{k-2}$, this is simply
\begin{align*}
% &-\frac{\frac{n-k+2}{j}}{j(k-2)!}\int_0^1t^{\frac{n-k+2}{j}-1}(1-t)^{k-2} \\
-\frac{\frac{n-k+2}{j}}{j(k-2)!}\frac{\Gamma(k-1)\Gamma(\frac{n-k+2}{j})}{\Gamma(\frac{n-k+2}{j}+k-1)}
=\frac{-\Gamma\left(1+\frac{n-k+2}{j}\right)}{j\Gamma\left(k-1+\frac{n-k+2}{j}\right)}.
\end{align*}

% To prove this, we consider $$a_m=\frac{(-1)^m}{(m-1)!(k-2-m)!(n-k+2+mj)}$$ and 
% $$S_m=F(\cdot) - \frac{\Gamma\left(1+\frac{n-k+2}{j}\right)}{j\Gamma\left(k-1+\frac{n-k+2}{j}\right)}$$ for $1\leq m\leq k-2$ and $S_0=0$, where $F(\cdot)$ stands for the generalized hypergeometric series. Notice that $S_{k-2}=-\Gamma\left(1+\frac{n-k+2}{j}\right)j^{-1}\Gamma\left(k-1+\frac{n-k+2}{j}\right)^{-1}$, so that it suffices to prove that $S_m-S_{m-1}=a_m$ for all $1\leq m\leq k-2$. Indeed, we then have
% \begin{equation*}
%     \sum_{i=1}^{k-2}\frac{(-1)^i}{(i-1)!(k-2-i)!(n-k+2+ij)}=\sum_{m=1}^{k-2}a_m=S_{k-2},
% \end{equation*}
% since the last sum telescopes. We calculate
% \begin{align*}
%     S_m-S_{m-1}&=...
%     \\
% \end{align*}

\end{proof}

With the previous Lemma in mind, we can proceed with the proof of Theorem $1$.
\subsection{Proof of Theorem 1}
By Lemma \ref{lem triangle},
$$\mathbb{P}_{k,n}=\mathbb{P}(Y_{(n)} \leq Y_{(1)}+\cdots+Y_{(k-1)}).$$
Using the law of total probability, we may write it as the $k-1$-iterated integral of 
\begin{multline*}
\mathbb{P}\left(Y_{(n)}\leq x_1+\cdots + x_{k-1} \mid (Y_{(1)},\ldots,Y_{(k-1)})=(x_1,\ldots,x_{k-1})\right) \\
    \times f_{(Y_{(1)},\ldots,Y_{(k-1)})}(x_1,\ldots,x_{k-1})    
\end{multline*} 
over the region $0<x_1\leq\ldots\leq x_{k-1}<\infty$.
We purposely choose to set up the bounds as
    \begin{equation*}
    \int_0^{\infty}\int_0^{x_{k-2}}\cdots\int_0^{x_2}\int_{x_{k-2}}^{\infty} \,d{x_{k-1}}\,d{x_1}\cdots  d{x_{k-2}}.
\end{equation*}
To evaluate the previous conditional probability, it is useful to use Lemma \ref{lem main} to interpret the $Y_i$ as given by $n-1$ points randomly chosen on the interval $(0,W_n)$, which produces $n$ segments. If we know that $Y_{(i)}=x_i$ for $1\leq i\leq k-1$, we might as well place $x_1,\ldots,x_{k-1}$ and $x_1+\cdots+x_{k-1}$ on the interval $(0,W_n)$. Now consider picking $n-k$ new points at random in $(x_{k-1},W_n)$ to produce $n-k+1$ segments, but look at them as $n-k+1$ i.i.d. exponential random variables with mean $1$, say $Y'_1,\ldots,Y'_{n-k+1}$. Now it is easy to see that $\max\{Y'_i\}$ falling in the interval $(x_{k-1},x_1+\cdots+x_{k-1})$ is equivalent to the event $Y_{(n)}\leq x_1+\cdots+x_{k-1}$ knowing that $Y_{(i)}=x_i$ for $1\leq i\leq k-1$. Indeed, the length of the interval is $x_1+\cdots+x_{k-2}$ and we have that $Y_{(n)}$ can be at most $$Y_{(k-1)}+\max\{Y'_i\}<x_1+\ldots+x_{k-1}$$ if $\max\{Y'_i\}$ is smaller than $x_1+\cdots+x_{k-2}$, while conversely 
$$\max\{Y'_i\}<Y_{(n)}-x_{k-1}<x_1+\cdots+x_{k-2}.$$ 
But $\max\{Y'_i\}$ is in $(x_{k-1},x_1+\cdots+x_{k-1})$ if and only if all the new segments fall in $(x_{k-1},x_1+\cdots+x_{k-1})$, or equivalently in $(0,x_1+\cdots+x_{k-2})$. Thus,
\begin{multline*}
    \mathbb{P}\left(Y_{(n)}\leq x_1+\cdots + x_{k-1} \mid (Y_{(1)},\ldots,Y_{(k-1)})=(x_1,\ldots,x_{k-1})\right)\\=\mathbb{P}(Y'_{1}\leq x_1+\cdots+x_{k-2},\ldots,Y'_{n-k+1}\leq x_1+\cdots+x_{k-2}).
\end{multline*}
Since the variables $Y'_i$ are i.i.d. exp($1$), we might as well write it as
$(\mathbb{P}(Y'_1\leq x_1+\cdots+x_{k-2}))^{n-k+1}=\left(1-e^{-(x_1+\cdots+x_{k-2})}\right)^{n-k+1}$.
This result is reminiscent of the memoryless property of the exponential distribution.

Also, it is a well-known result that the joint distribution of the $n$ order statistics of $n$ i.i.d. and continuous random variables with density $f(x)$ and distribution function $F(x)$ is given by
    $$f_{X_{(1)},\ldots,X_{(n)}}(x_1,\ldots,x_n)=n!\prod_{i=1}^nf(x_i),$$
    and one can integrate with respect to $dx_{j+1}\,dx_{j+2}\ldots dx_{n}$ to get the joint distribution of the first $j$ order statistics, namely
    $$f_{X_{(1)},\ldots,X_{(j)}}(x_1,\ldots,x_j)=n!\frac{(1-F(x_j))^{n-j}}{(n-j)!}\prod_{i=1}^{j}f(x_i).$$
Then we find that
    \begin{align*}
        f_{(Y_{(1)},\ldots,Y_{(k-1)})}(x_1,\ldots,x_{k-1})&=\frac{n!}{(n-(k-1))!}(1-(1-e^{-x_{k-1}}))^{n-(k-1)}e^{-x_1}\cdots e^{-x_{k-1}}\\
        &=n(n-1)\cdots (n-k+2)e^{-(x_1+\cdots+x_{k-2}+(n-k+2)x_{k-1})}.
    \end{align*}
 In other words, $\mathbb{P}_{k,n}$ is given by
    \begin{align*}
    \begin{split}
            &n(n-1)\cdots (n-k+2)\int_0^{\infty}\int_0^{x_{k-2}}\cdots\int_0^{x_2}\int_{x_{k-2}}^{\infty}
            \left(1-e^{-(x_1+\cdots+x_{k-2})}\right)^{n-k+1}\\
            &\times e^{-(x_1+\cdots+x_{k-2}+(n-k+2)x_{k-1})} \,d{x_{k-1}}\,d{x_1}\cdots  d{x_{k-2}}
            \end{split}\\
            \begin{split}
            &=n(n-1)\cdots (n-k+3)\int_0^{\infty}e^{-(n-k+2)x_{k-2}}\int_0^{x_{k-2}}\cdots\int_0^{x_2}
            \left(1-e^{-(x_1+\cdots+x_{k-2})}\right)^{n-k+1}\\
            &\times e^{-(x_1+\cdots+x_{k-2})} \,d{x_1}\cdots d{x_{k-2}}.
            \end{split}
        \end{align*}
Since the inner integrand is the derivative of $(1-e^{-(x_1+\cdots+x_{k-2})})^{n-k+2}$ up to a factor of $(n-k+2)^{-1}$, we get that $\mathbb{P}_{k,n}$ is equal 
to 
\begin{multline*}
\frac{n(n-1)\cdots(n-k+3)}{n-k+2}\int_0^{\infty}e^{-(n-k+2)x_{k-2}}\int_0^{x_{k-2}}\cdots\int_0^{x_3} 
\left(1-e^{-(2x_2+x_3+\cdots+x_{k-2})}\right)^{n-k+2}\\-\left(1-e^{-(x_2+\cdots+x_{k-2})}\right)^{n-k+2}
\,d{x_2}\cdots d{x_{k-2}}.
\end{multline*}
Now, we apply the binomial theorem to expand the inner integrand as 
\begin{equation*}
    \sum_{j=0}^{n-k+2}\binom{n-k+2}{j}(-1)^j\left(e^{-j(2x_2+x_3+\cdots+x_{k-2})}-e^{-j(x_2+\cdots+x_{k-2})}\right).
\end{equation*}
Noticing that when $j=0$, the summand equals $0$, we get that
\begin{multline*}
    \mathbb{P}_{k,n}=\frac{n(n-1)\cdots(n-k+3)}{n-k+2}
    \sum_{j=1}^{n-k+2}\binom{n-k+2}{j}(-1)^j 
    \int_0^{\infty}e^{-(n-k+2)x_{k-2}}\\ \times
    \int_0^{x_{k-2}}\cdots\int_0^{x_3} 
    e^{-j(2x_2+x_3+\cdots+x_{k-2})}-e^{-j(x_2+\cdots+x_{k-2})}
    \,d{x_2}\cdots d{x_{k-2}}.
\end{multline*}
Now the result follows from Lemma \ref{lemma}, as long as $k\geq 4$. To finish the proof, we show directly that 
\begin{equation} \label{Appendix}
\mathbb{P}_{3,n}=\binom{2n-2}{n}^{-1}=
\frac{n}{n-1}\sum_{j=1}^{n-1}\frac{(-1)^{j+1}\binom{n-1}{j}}{\frac{n-1}{j}+1},
\end{equation}
thus establishing the result for $k\geq 3$. 
The proof of the first equality in \eqref{Appendix} is still interesting in its own right as a simple application of our method for a fixed value of $k$. Since the setup of the method still holds for $k=3$, we may write
\begin{align*}
            \mathbb{P}_{3,n}
            =\int_0^{\infty}\int_0^v P\left(Y_{(n)}\leq u + v \mid (Y_{(1)},Y_{(2)})=(u,v)\right)f_{(Y_{(1)},Y_{(2)})}(u,v)dudv.
\end{align*}
Then proceeding as we did at the beginning of the proof of Theorem $1$, we have
\begin{align*}
    &\int_0^{\infty}\int_0^v P\left(Y_{(n)}\leq u + v \mid (Y_{(1)},Y_{(2)})=(u,v)\right)f_{(Y_{(1)},Y_{(2)})}(u,v)dudv \\
            &=\int_0^{\infty}\int_0^v (1-e^{-u})^{n-2}n(n-1)e^{-u-(n-1)v}dudv,
\end{align*}
which can be integrated directly as follows.
\begin{align*}
            \mathbb{P}_{3,n}
            &=\int_0^{\infty}ne^{-(n-1)v}\int_0^v(n-1)e^{-u}(1-e^{-u})^{n-2}dudv \\
            &=\int_0^{\infty}ne^{-(n-1)v}(1-e^{-v})^{n-1}dv \\
            &=n\int_0^1t^{n-2}(1-t)^{n-1}dt \\
            &=nB(n-1,n)\\
            &=\frac{(n-2)!n!}{(2n-2)!} \\
            &=\binom{2n-2}{n}^{-1},
        \end{align*}
    where in the fourth to last equality we made the substitution $t=e^{-v}$, and where $B(\cdot,\cdot)$ stands for the Beta function. Now it remains to show the second equality in \eqref{Appendix}. 
    Observe that 
    $$
    \frac{j}{n-1+j}=1-(n-1)\int_{0}^1t^{n-1+j-1}\,dt,
    $$
    from which the right hand side of $\eqref{Appendix}$ is simply
    \begin{align*}
    n\sum_{j=0}^{n-1}(-1)^j\binom{n-1}{j}\int_{0}^1t^{n-1+j-1}\,dt
    =n\int_0^1t^{n-2}(1-t)^{n-1}dt.
    \end{align*}
\begin{remark}It is interesting to write $\mathbb{P}_{3,n}$ as
$nB(n-1,n)$, because it should be clear that this result could be generalized to more values of $k$. As a matter of fact, it can be shown in the same way that $$\mathbb{P}_{4,n}=\frac{n(n-1)}{n-2}\left(\frac{1}{2} B\left(n-1,\frac{n-2}{2}\right)-B\left(n-1,n-2\right) \right)$$ 
and that \begin{multline*}\mathbb{P}_{5,n}=
\frac{n(n-1)(n-2)}{(n-3)^2}\\\times\left(\frac{1}{2}B\left(n-2,\frac{n-3}{3}\right)+\frac{1}{2}B(n-2,n-3)-B\left(n-2,\frac{n-3}{2}\right)\right).\end{multline*}
It is not clear what a general formula for $\mathbb{P}_{k,n}$ using Beta functions would look like yet, and when $k\geq 6$, the calculations become much more tedious and require switching the bounds of integration of at least $4$ iterated integrals. In hindsight, the formula given by Theorem $1$ is simpler to obtain but also simpler in its formulation. The only benefit a formula involving linear combinations of Beta functions would bring is potentially in its computational stability (notice that the number of Beta functions does not vary with $n$, but their value does, while the number of terms we sum over in our formula grows with $n$).
\end{remark}

    % The main goal of our paper was to prove the result involving $k$-gons, but to show that our method could be applied to other problems, we consider the problem of forming no triangles, as it was stated in the Introduction. Namely, we prove the result we stated using our method. Fear not, it will amount to a well chosen triangle inequality.
    
    % In hopes that this will make you want to try it out with similar problems, perhaps even finding the probability that one cannot form a $k$-gon from an $n$ piece broken stick, which probably involves Fibonacci numbers for all $k$ (or does it?).

    % We end this paper by stating interesting results that can be just as easily obtained using our method...

\section{Concluding remarks}
The goal of our paper was first and foremost to expose this new method to solve broken stick probability problems. We applied it to one of these problems that seemed out of reach, but it would be interesting to apply it to many more. For example, finding the probability that there exist $k$ segments of an $n$-piece broken stick that can form a $k$-gon is an intriguing question, perhaps because the only known result is for $k=3$ and involves Fibonacci numbers, as it was stated in the Introduction. Our method could be applied in almost the same way, but Lemma \ref{lem triangle} would need to be changed. For instance, one would need to look at the union of the events $$\Delta_{(k+j)} \leq \sum_{i=1}^{k-1}\Delta_{(i+j)}, \qquad j=0,\ldots,n-k.$$ 
Another nice generalization that comes to mind immediately is that of forming a tetrahedron from a stick broken up into $6$ segments. Surely our method still applies, but it is a direct application of the generalized triangle inequality that fails us again. Known conditions to form polyhedra seem much harder to handle in practice. This problem is particularly interesting for its link with distance geometry, as it was noted in \cite{page_calculus_nodate}. They raised the following classic question in the field:
    \textit{Given $\binom{n}{2}$ positive numbers, can one find n points in space whose pairwise distances are the given numbers?} 
    This question has direct applications in wireless sensor networks and subjects like molecular biology, and the broken stick problems in $2$ or $3$ dimensions can be seen as random versions of this question
    which could even be extended to problems about pairwise distances of points in $R^m$ being the lengths of broken sticks.

% The author (and collaborator(s)) is/are also investigating a discrete analogue that could help solve these kind of problems by reducing them to partition analysis problems.    

\section*{Acknowledgments}
The author would like to thank Claude Bélisle, whose article \cite{belisle_polygon_2011} used similar ideas, for suggesting an approach to this problem.

\bibliographystyle{abbrv}
\bibliography{references.bib}

\end{document}